\documentclass{article}

\usepackage{amsmath,amsthm,amssymb,mathrsfs,amstext, enumitem,comment}

\newtheorem{thm}{Theorem}[section]
\newtheorem{lem}[thm]{Lemma}

\theoremstyle{definition}
\newtheorem{defn}[thm]{Definition}
\newtheorem{example}[thm]{Example}

\newtheorem{problem}[thm]{Question}
\newtheorem{cor}[thm]{Corollary}
\newtheorem{conjecture}[thm]{Conjecture}
\theoremstyle{remark}
\newtheorem{rem}[thm]{Remark}

\numberwithin{equation}{section}
\begin{document}

\title{Product of difference sets of set of primes }


\author{Sayan Goswami \\ The Institute of Mathematical Sciences\\ A CI of Homi Bhabha National Institute\\ CIT Campus, Taramani, Chennai 600113, India}
\maketitle







\begin{abstract}
In a recent work \cite{key-11}, A. Fish proved that if $E_{1}$ and
$E_{2}$ are two subsets of $\mathbb{Z}$ of positive upper Banach
density, then there exists $k\in\mathbb{Z}$ such that $k\cdot\mathbb{Z}\subset\left(E_{1}-E_{1}\right)\cdot\left(E_{2}-E_{2}\right).$
In this article we will show that a similar result is true for the
set of primes $\mathbb{P}$ (which has density $0$). We will prove
that there exists $k\in\mathbb{N}$ such that $k\cdot\mathbb{N}\subset\left(\mathbb{P}-\mathbb{P}\right)\cdot\left(\mathbb{P}-\mathbb{P}\right),$
where $\mathbb{P}-\mathbb{P}=\left\{ p-q:p>q\,\text{and}\,p,q\in\mathbb{P}\right\} .$
\end{abstract}

\section{Introduction}

Estimating the size of sums and products of subets of $\mathbb{Z}$
are widely studied in Additive combinatorics. For any two sets $A,B\subset\mathbb{Z}\,\text{(or }\mathbb{R}\text{)}$,
define $A+B$ and $A\cdot B$ as $A+B=\left\{ a+b:a\in A,b\in B\right\} $
and $A\cdot B=\left\{ a\cdot b:a\in A,b\in B\right\} .$ The following
conjecture of Erd\H{o}s and Szemer\'{e}di \cite{key-9.1} is one
of the central problem on sum-product estimates. 
\begin{conjecture}
\cite{key-9.1} If $A\subset\mathbb{Z}\text{(or }\mathbb{R}\text{)}$
is any finite set, then for every $\epsilon>0$, we should have 
\[
\vert A+A\vert+\vert A\cdot A\vert\gg\vert A\vert^{2-\epsilon}.
\]
\end{conjecture}

The best known upper bound till date is due to Konyagin and Shkredov
\cite{key-30}, and the proof is based on a breakthrough paper of
Solymosi \cite{key-13}, saying that
\[
\vert A+A\vert+\vert A\cdot A\vert\gg\vert A\vert^{\frac{4}{3}+c},
\]
where $c<\frac{5}{9813}.$ In \cite{key-11}, A. Fish asked the following
question, which is a twisted version of the above problems. 
\begin{problem}
\cite[Question 1]{key-11} For a given infinite set $E\subset\mathbb{Z},$
how much structure does possess the set $\left(E-E\right)\cdot\left(E-E\right)?$
\end{problem}

Before we proceed further, let us recall the definition of the upper
Banach density of a set.
\begin{defn}
\textbf{(Upper Banach density)}
\begin{enumerate}
\item For any set $E\subseteq\mathbb{Z}$, the upper Banach density of $E$
is 
\[
d^{\star}\left(E\right)=\limsup_{N\rightarrow\infty}\sup_{a\in\mathbb{Z}}\frac{\vert E\cap\left\{ a,a+1,\ldots,a+(N-1)\right\} \vert}{N}.
\]
 
\item For any set $E\subseteq\mathbb{N}$, the upper Banach density of $E$
is 
\[
d^{\star}\left(E\right)=\limsup_{N\rightarrow\infty}\sup_{a\in\mathbb{N}}\frac{\vert E\cap\left\{ a,a+1,\ldots,a+(N-1)\right\} \vert}{N}.
\]
\end{enumerate}
\end{defn}

In \cite{key-31}, M. Bj\H{o}rklund and A. Fish proved that for any
three set $A,B$ and $C$ with positive upper Banach density, there
exists $k\geq1$ such that $k\cdot\mathbb{Z}\subseteq\left(A-A\right)\cdot\left(B-B\right)-\left(C-C\right)^{2}$, which was improved by A. Fish in \cite{key-11}, that we will discuss in subsection 1.1.

H. Furstenberg \cite{key-2} found a connection between the difference
of sets of positive upper Banach density and the return time of a
set with positive measure in measurable dynamical system. Before further discussion,  let us recall some preliminaries from measurable dynamical
system.

\subsection{Ergodic foundation}

Let $\left(X,\mathcal{B},\mu,T\right)$ be a measure preserving system
and $A\subseteq X$ be a measurable set. Denote by $R\left(A\right)$
the set of return time defined as  
$$R\left(A\right)=\left\{ n\in\mathbb{Z}:\mu\left(A\cap T^{-n}A\right)>0\right\}. $$
For any $A\subseteq X$ with $\mu\left(A\right)>0,$ we know by Poincar\'{e}'s
theorem that $R\left(A\right)\neq\emptyset.$ As a consequence of
Furstenberg correspondence principle \cite{key-2}, we have, for any subset $E\subseteq \mathbb{Z},$ with $d^{\star}\left(E\right)>0$,
 there exists a measurable dynamical system $\left(X,\mathcal{B},\mu,T\right)$
such that $A\in\mathcal{B}$ such that $d^{\star}\left(E\right)=\mu\left(A\right)$
and $R\left(A\right)\subseteq E-E$. In \cite{key-11}, A. Fish proved
the following theorem.
\begin{thm}
\label{fishtheo} \cite[Theorem 1.1]{key-11} Let $\left(X,\mu,T\right)$ and $\left(Y,\nu,S\right)$
be two measure preserving system and let $A\subset X,$ $B\subset Y$
be two measurable sets with $\mu\left(A\right)>0$ and $\nu\left(B\right)>0.$
Then there exist $k\in\mathbb{Z}$ such that $k\cdot\mathbb{Z}\subseteq R\left(A\right)\cdot R\left(B\right)$.
\end{thm}

As a consequence of Theorem \ref{fishtheo}, and Furstenberg correspondence
principle, we have the following corollary.
\begin{cor}\label{iimm}
If $E_{1}$ and $E_{2}$ are two subsets of positive density of $\mathbb{Z}$
, then there exists $k\in\mathbb{Z}$ such that $k\cdot\mathbb{Z}\subseteq\left(E_{1}-E_{1}\right)\cdot\left(E_{2}-E_{2}\right).$ 
\end{cor}
The Corollary \ref{iimm} immediately improves the result of M. Bj\H{o}rklund and A. Fish \cite{key-31}.

Let $\mathcal{P}_{f}\left(\mathbb{N}\right)$ be the collection of
nonempty finite subsets of $\mathbb{N}.$ The following notion of
$IP$, $IP_{r}$ and $\Delta_{r}$ sets will be necessary in our work.
\begin{defn}
\label{defn ip} ($IP$, $IP_{r}$ and $\varDelta_{r}$ sets) If $\left(S,+\right)$
be a commutative semigroup, then
\begin{enumerate}
\item for any sequence $\langle y_{n}\rangle_{n\in\mathbb{N}}$, let 
\[
FS\left(\langle y_{n}\rangle_{n\in\mathbb{N}}\right)=\left\{ \sum_{t\in H}y_{t}:H\in\mathcal{P}_{f}\left(\mathbb{N}\right)\right\} ,
\]
\item for any $n\in\mathbb{N}$ and sequence $\langle y_{i}\rangle_{i=1}^{n}$,
let 
\[
FS\left(\langle y_{i}\rangle_{i=1}^{n}\right)=\left\{ \sum_{t\in H}y_{t}:H\subseteq\left\{ 1,2,\ldots,n\right\} \right\} .
\]
\item A set $A$ is said to be an $IP$ set if there exists a sequence $\langle x_{n}\rangle_{n\in\mathbb{N}}$
such that $A=FS\left(\langle x_{n}\rangle_{n\in\mathbb{N}}\right).$ 
\item A set $A$ is said to be an $IP_{r}$ set if there exists a sequence
$\langle y_{i}\rangle_{i=1}^{r}$ such that $A=FS\left(\langle y_{i}\rangle_{i=1}^{r}\right).$
\item A set $A\subset\mathbb{N}$ is said to be a $\varDelta_{r}$ set if
there exists a set $S\subset\mathbb{N}$ with $\vert S\vert=r$ such
that $A=\left\{ s-t:s>t\text{ and }s,t\in S\right\} .$
\end{enumerate}
\end{defn}

Suppose $\mathcal{F}$ is a family of sets. A set $A$ is said to
be $\mathcal{F}^{\star}$ set if $A\cap F\neq\emptyset$ for all $F\in\mathcal{F}.$
We will call a set is $IP^{\star}$( $IP_{r}^{\star}$ and $\varDelta_{r}^{\star}$ resp.)
if it intersects all the $IP$ sets ( $IP_{r}$ sets and $\varDelta_{r}$
sets resp.). Note that every $IP_{r}$ set contains $\varDelta_{r}$
set. To check this, let $FS\left(\langle x_{n}\rangle_{n=1}^{r}\right)$
be an $IP_{r}$ set and let
\[
S=\left\{ x_{1},x_{1}+x_{2},\ldots,x_{1}+x_{2}+\cdots+x_{n}\right\} ,
\]
Now $FS\left(\langle x_{n}\rangle_{n=1}^{r}\right)$ contains elements
of the form $\left\{ s-t:s>t\text{ and }s,t\in S\right\} $. Hence
every $\varDelta_{r}^{\star}$ set is $IP_{r}^{\star}$.

For details the readers may see the book \cite{key-29}. 
\begin{example}
\label{returntimesip*} For any set $A\subseteq X$ with $\mu\left(A\right)>0,$
let $r=\frac{1}{[\mu\left(A\right)]}+1.$ As $\mu\left(\cup_{i=1}^{r}T^{-i}\left(A\right)\right)\leq1$,
there exist distinct $i,j\in\left\{ 1,2,\ldots,r\right\} $ with $i>j$
such that $\mu\left(T^{-i}\left(A\right)\cap T^{-j}\left(A\right)\right)>0,$
i.e., $i-j\in R\left(A\right).$ Hence $R\left(A\right)$ is a  $\varDelta_{r}^{\star}$ set
and so an $IP_{r}^{\star}$ set.
\end{example}

Before proceed let us recall some conjectures and basic results on
$\mathbb{P}-\mathbb{P}$.

\subsection{A brief introduction to $\mathbb{P}-\mathbb{P}$ }

In $1905,$ Maillet \cite{key-33} conjectured that the set of the
difference of primes should contain all even numbers.
\begin{conjecture}
\cite{key-33} Every even number is the difference of two primes.
\end{conjecture}

Originally before Maillet, there were two stronger forms of this conjecture.
In $1901$, Kronecker \cite{key-12} made the following conjecture.
\begin{conjecture}
\cite{key-12} Every even number can be expressed in infinitely many
ways as the difference of two primes.
\end{conjecture}

In $1849$, Polignac \cite{key-18} conjectured the following which
is the most general one.
\begin{conjecture}
\cite{key-18} Every even number can be written in infinitely many
ways as the difference of two consecutive primes.
\end{conjecture}

Based on \cite{key-19}, Zhang \cite{key-21} made a recent breakthrough
and proved that there exists an even number not more than $7\times107$
which can be expressed in infinitely many ways as the difference of
two primes. Soon after, Maynard and Tao \cite{key-26,key-27} reduced
the limit of such an even number to not more than $600$. The best
known result now is not more than $246$; for details see \cite{key-26}.
The following theorem is due to Huang and Sheng Wu \cite{key-3}, an outstanding application of Pigeonhole principle and 
Zhang-Maynard-Tao theorem \cite[Theorem 3.1]{key-3}. 
\begin{thm}
\label{p-p} There exists $r\in\mathbb{N}$ such that $\mathbb{P}-\mathbb{P}$
is an $\varDelta_{r}^{\star}$ set. 
\end{thm}

Hence it is an $IP_{r}^{\star}$ set. As the set of primes has density
$0,$ we can't say nothing about $\left(\mathbb{P}-\mathbb{P}\right)\cdot\left(\mathbb{P}-\mathbb{P}\right)$
. 

We will use theorem \ref{p-p} to
deduce the following result.
\begin{thm}
\label{p-pmain} There exists $k\in\mathbb{N}$ such that $k\cdot\mathbb{N}\subseteq\left(\mathbb{P}-\mathbb{P}\right)\cdot\left(\mathbb{P}-\mathbb{P}\right).$ 
\end{thm}

Let us recall some basic preliminaries of algebra of ultrafilters,
which will be helpfull for us.

\subsection{A brief review of topological algebra}

In this subsection we will recall some basic preliminaries of the algebra
of ultrafilters, which we will use to deduce some corollaries. For
details the readers can see the beautiful book on algebra of ultrafilters
\cite{key-29} and a short review \cite[Chapter 2]{key-30}.  Denote by $\beta\mathbb{N}$,
the Stone-\v{C}ech compactification of $\mathbb{N}$. It can
be shown that $\beta\mathbb{N}$ is the set of all ultrafilters over
$\mathbb{N}$, where the points of $\mathbb{N}$ are identified with
the principle ultrafilters. The basis for the topology is $\left\{ \bar{A}:A\subseteq\mathbb{N}\right\} $,
where $\bar{A}=\left\{ p\in\beta\mathbb{N}:A\in p\right\} $. The
operation of $\mathbb{N}$ can be extended to $\beta\mathbb{N}$ making
$\left(\beta\mathbb{N},+\right)$ a compact, right topological semigroup.
For $p,q\in\beta\mathbb{N}$
and $A\subseteq\mathbb{N}$, $A\in p+q$ if and only if $\left\{ x\in\mathbb{N}:-x+A\in q\right\} \in p$,
where $-x+A=\left\{ y\in\mathbb{N}:x+y\in A\right\} $. In \cite{key-28},
Ellis proved that every compact right topological semigroup contains
idempotents. Note if $A\in p=p+p$, there exists a sequence $\langle x_{n}\rangle_{n\in\mathbb{N}}$
such that $FS\left(\langle x_{n}\rangle_{n\in\mathbb{N}}\right)\subseteq A.$
In fact the converse is also true. That means if $A$ contains an
$IP$ set, then there exists an idempotent $p$ such that $A\in p.$
So, a set $A$ is $IP^{\star}$ if and only if $A\in p$ for all idempotents
$p\in\beta\mathbb{N}$.
\begin{rem}
\label{intersect} For any $n\in\mathbb{N},$ let $A_{1},A_{2},\ldots,A_{n}$
be $IP^{\star}$ sets. Then for each idempotents $p\in\beta\mathbb{N}$,
$A_{i}\in p$ for all $i\in\left\{ 1,2,\ldots,n\right\} .$ So, $\cap_{i=1}^{n}A_{i}\in p$
for all idempotents $p\in\beta\mathbb{N}.$ Hence $\cap_{i=1}^{n}A_{i}$
is an $IP^{\star}$ set.
\end{rem}

The following theorem is our main result.
\begin{thm}
\label{main theorem} Let $r\in\mathbb{N}$ and let $A,B\subseteq\mathbb{N}$
be $IP^{\star}$ set and $IP_{r}^{\star}$ sets respectively. Then
there exists $k\in A$ such that $k\cdot\mathbb{N}\subseteq A\cdot B.$ 
\end{thm}

\section{Our results}

The following lemma will be necessary for the proof of our main theorem.
\begin{lem}
\label{ipintersect} Let $m\in\mathbb{N}$ and let $A\subseteq\mathbb{N}$
be an $IP^{\star}$ set. Then $m\cdot A$ is also an $IP^{\star}$
set.
\end{lem}

\begin{proof}
Let $\langle x_{n}\rangle_{n\in\mathbb{N}}$ be any sequence. For
each 
$i\in\left\{ 0,1,\ldots,m-1\right\}$, let 
$$x_{i}\equiv i^{\prime}(\mod m),$$
where $i^{\prime}\in\left\{ 0,1,\ldots,m-1\right\} .$ Now pick $H_{1}$ (consider
$H_{1}$ to be the collection of those $i$'s such that all $i^{\prime}$s are same) such that $m\vert\sum_{t\in H_{1}}x_{t}.$ Now
continue this process to obtain a disjoint sequences $\langle H_{n}\rangle_{n\in\mathbb{N}}$
of the finite subsets of $\mathbb{N}$ such that $m\vert\sum_{t\in H_{n}}x_{t}$
for each $n\in\mathbb{N}$. Now choose a new sequence $\langle y_{n}\rangle_{n\in\mathbb{N}}$
such that $y_{n}=\frac{1}{m}\sum_{t\in H_{n}}x_{t}$ for each $n\in\mathbb{N}$.
Then $A\cap FS\left(\langle y_{n}\rangle_{n\in\mathbb{N}}\right)\neq\emptyset$
and this implies $m\cdot A\cap FS\left(\langle x_{n}\rangle_{n\in\mathbb{N}}\right)\neq\emptyset$.
This proves the lemma.
\end{proof}
Now we are ready to prove our main theorem.
\begin{proof}[\textbf{Proof of Theorem \ref{main theorem}:}]
 Let $B$ be an $IP_{r}^{\star}$ set. Then for any $x\in\mathbb{N}$,
there exists $p\left(x\right)\in FS\left(\left\{ 1,2,\ldots,r\right\} \right)$
such that $x\cdot p\left(x\right)\in B.$ Now from Lemma \ref{ipintersect},
$m\cdot A$ are also $IP^{\star}$ sets for each $m\in FS\left(\left\{ 1,2,\ldots,r\right\} \right)$.

Now choose $k\in A\cap\bigcap_{m\in FS\left(\left\{ 1,2,\ldots,r\right\} \right)}m\cdot A.$
Then $\frac{k}{m}\in A$ for each $m\in FS\left(\left\{ 1,2,\ldots,r\right\} \right)$,
which implies $\frac{k}{p\left(x\right)}\in A$ for each $x\in\mathbb{N}$.
Hence $k\cdot x=\frac{k}{p\left(x\right)}\cdot xp\left(x\right)\in A\cdot B$
for each $x\in\mathbb{N}$. 

This completes the proof.
\end{proof}
Now from Theorem \ref{p-p} and Theorem \ref{main theorem}, we have
our main Theorem \ref{p-pmain}. From Example \ref{returntimesip*},
we have the following interesting result.
\begin{cor}
If $A\subseteq\mathbb{N}$ is a set of positive density, then there
exists $k_{1}\in\mathbb{N}$ such that $k_{1}\cdot\mathbb{N}\subseteq\left(\mathbb{P}-\mathbb{P}\right)\cdot\left(A-A\right).$
\end{cor}

A close analysis of proof of Theorem \ref{main theorem} shows that
\[
C=A\cap\bigcap_{m\in FS\left(\left\{ 1,2,\ldots,r\right\} \right)}m\cdot A
\]
 is an $IP^{\star}$ set (from Remark \ref{intersect}), i.e. $C\in p$ for all idempotent $p\in\beta\mathbb{N}.$
Then there exists a sequence $\langle x_{n}\rangle_{n\in\mathbb{N}}$
such that $FS\left(\langle x_{n}\rangle_{n\in\mathbb{N}}\right)\subset C\in p.$
Hence we have the following corollary, which generalizes all the previous
results. 
\begin{cor}
Let $A,B\subseteq\mathbb{N}$ be two subsets of positive density and
let $\mathbb{P}$ be the set of primes. Then the following holds
\begin{enumerate}
\item there exists a sequences $\langle x_{n}\rangle_{n\in\mathbb{N}}$
such that
\[
FS\left(\langle x_{n}\rangle_{n\in\mathbb{N}}\right)\cdot\mathbb{N}\subseteq\left(A-A\right)\cdot\left(B-B\right),
\]
\item there exists a sequences $\langle y_{n}\rangle_{n\in\mathbb{N}}$
such that 
\[
FS\left(\langle y_{n}\rangle_{n\in\mathbb{N}}\right)\cdot\mathbb{N}\subseteq\left(\mathbb{P}-\mathbb{P}\right)\cdot\left(A-A\right),
\]
and
\item there exists a sequences $\langle z_{n}\rangle_{n\in\mathbb{N}}$
such that 
\[
FS\left(\langle z_{n}\rangle_{n\in\mathbb{N}}\right)\cdot\mathbb{N}\subseteq\left(\mathbb{P}-\mathbb{P}\right)\cdot\left(\mathbb{P}-\mathbb{P}\right).
\]
\end{enumerate}
\end{cor}

\subsection*{Acknowledgments}

We are very thankful to the anonymous referees for their helpful comments
on the previous draft of this manuscript.


\end{document}